\documentclass[12pt]{extarticle}

\usepackage[T1]{fontenc}
\usepackage{amsmath}
\usepackage{amssymb}
\usepackage{amsthm}
\usepackage{tikz-cd}
\usepackage{theoremref}
\usepackage[hidelinks]{hyperref}
\usepackage{cleveref}
\usepackage{enumitem}
\usepackage{abstract}
\usepackage[margin=3cm]{geometry}

\newtheorem{theorem}{Theorem}[section]
\newtheorem*{theorem*}{Theorem}
\newtheorem{proposition}{Proposition}[section]

\newtheorem{lemma}{Lemma}[section]
\newtheorem{corollary}{Corollary}[section]
\theoremstyle{definition}
\newtheorem{example}{Example}[section]
\newtheorem{remark}{Remark}[section]

\usepackage{titlesec}
\titleformat*{\section}{\large\bfseries}
\titleformat*{\subsection}{\normalsize\bfseries}

\newcommand{\cd}{\mathcal{D}}

\newcommand{\cf}{\mathcal{F}}
\newcommand{\cg}{\mathcal{G}}
\newcommand{\ch}{\mathcal{H}}
\newcommand{\ci}{\mathcal{J}}

\newcommand{\cl}{\mathcal{L}}

\newcommand{\cn}{\mathcal{N}}
\newcommand{\co}{\mathcal{O}}

\newcommand{\cq}{\mathcal{Q}}
\newcommand{\cs}{\mathcal{S}}

\newcommand{\cz}{\mathcal{Z}}

\newcommand{\bp}{\mathbf{P}}
\newcommand{\bg}{\mathbf{G}}
\newcommand{\bc}{\mathbf{C}}
\newcommand{\bz}{\mathbf{Z}}

\newcommand{\fx}{\mathfrak{X}}

\newcommand{\E}{\mathrm{E}} 
\newcommand{\F}{\mathrm{F}} 
\newcommand{\G}{\mathrm{G}} 
\renewcommand{\H}{\mathrm{H}}
\newcommand{\HH}{\mathrm{HH}}

\renewcommand{\L}{\mathrm{L}}
\newcommand{\M}{\mathrm{M}}

\let\SS\S
\renewcommand{\S}{\mathrm{S}}
\newcommand{\T}{\mathrm{T}}

\newcommand{\V}{\mathrm{V}} 
\newcommand{\X}{\mathrm{X}} 
 
\newcommand{\Z}{\mathrm{Z}} 

\newcommand{\Aut}{\operatorname{Aut}}

\newcommand{\Ext}{\operatorname{Ext}}
\newcommand{\SExt}{\mathcal{E}xt}
\newcommand{\SEnd}{\mathcal{E}nd}
\newcommand{\Hilb}{\operatorname{Hilb}}

\newcommand{\End}{\operatorname{End}}

\newcommand{\Pic}{\operatorname{Pic}}
\newcommand{\Spec}{\operatorname{Spec}}
\newcommand{\SHom}{\mathcal{H}om}

\newcommand{\isom}{ \ \rotatebox[origin=c]{90}{$\thicksim$}}

\begin{document}
	
\title{\Large\textbf{Deformations of the Fano scheme of a cubic}}
\author{\normalsize SAMUEL STARK}
\date{}
\maketitle
	
\begin{abstract}
We study the deformation theory of the Fano scheme $\F=\F(\X)$ of lines on a cubic $\X$ of dimension $d$ 
with only finitely many singularities. By taking the relative Fano scheme, 
we define a morphism $\eta:\cd_{\X}\rightarrow\cd_{\F}$ of the local moduli functors associated to $\X$ and $\F$, respectively. 
We show that for $d\geqslant 5$, $\eta$ yields an isomorphism on first-order deformations; 
in particular, $\eta$ is an isomorphism whenever $\H^{0}(\Theta_{\X})=0$.
\end{abstract}

\section{Introduction}

Let $\bp$ be the complex projective space of dimension $d+1$, and $\X\subset\bp$ a cubic with a finite number of singularities. 
For $d\geqslant 3$, it is well-known that the geometry of $\X$ is largely determined by the Hilbert scheme $\F=\F(\X)$ of lines on $\X$, 
which is traditionally called the Fano scheme of $\X$. A great deal is known about $\F$ for $d=3$ or $d=4$ \cite{BeD, ClG, Fan, Huy, Weh}, 
and so our focus is on the $d\geqslant 5$ case, which has received much less attention. 
Altman and Kleiman \cite{AlK} show that $\F$ is an irreducible normal local complete intersection of dimension $2d-4$, 
and it is known that $\X$ can be recovered from $\F$ \cite{Cha}. 

In this paper, we relate the deformation theory of $\X$ to the deformation theory of $\F$. 
It is well-known that every infinitesimal deformation of $\X$ is given by a family of cubic hypersurfaces; 
by taking the relative Hilbert scheme, we define a morphism
\begin{equation*}
\eta:\cd_{\X}\rightarrow\cd_{\F}
\end{equation*}
of local moduli functors. A remarkable result of Beauville and Donagi \cite{BeD} asserts that if $\X$ is smooth of dimension $d=4$, 
the scheme $\F$ is deformation equivalent to the Hilbert scheme of two points of a K3 surface; in particular, 
there are deformations of $\F$ which are not induced by $\X$. In contrast, our main result is:
\begin{theorem*}\label{main}
Let $\X$ be cubic of dimension $d\geqslant 5$ having only finitely many singularities. The differential 
\begin{equation*}
d\eta:\Ext^{1}(\Omega^{1}_{\X}, \co_{\X})\rightarrow\Ext^{1}(\Omega^{1}_{\F}, \co_{\F})
\end{equation*}
of $\eta$ is an isomorphism. If $\H^{0}(\Theta_{\X})=0$, then $\eta$ is an isomorphism.
\end{theorem*}

Our proof relies on the standard description of $\F$ as a subscheme of the Grassmannian $\bg$ of lines in $\bp$. Parallel to $\eta$, there is a morphism
\begin{equation*}
\eta_{\ch}:\ch_{\X/\bp} \rightarrow \ch_{\F/\bg}
\end{equation*}
of local Hilbert functors, which is related to $\eta$ by a commutative square
\begin{equation*}
\begin{tikzcd}
\ch_{\X/\bp}  \arrow[r] \arrow[d, swap, "\eta_{\ch}"] & \cd_{\X}  \arrow[d, "\eta"]  \\
\ch_{\F/\bg} \arrow[r] & \cd_{\F},
\end{tikzcd}
\end{equation*}
where the horizontal morphisms are the forgetful ones. Consider the square
\begin{equation*}
\begin{tikzcd}
\H^{0}(\cn_{\X/\bp}) \arrow[r] \arrow[d, swap, "d\eta_{\ch}"] & \Ext^{1}(\Omega^{1}_{\X}, \co_{\X})  \arrow[d, "d\eta"]  \\
\H^{0}(\cn_{\F/\bg}) \arrow[r] & \Ext^{1}(\Omega^{1}_{\F}, \co_{\F}) 
\end{tikzcd}
\end{equation*}
of differentials. Relying on Borel-Bott-Weil computations and hypercohomology spectral sequences associated to the Koszul resolution of $\co_{\F}$, 
we show that the maps $\H^{0}(\cn_{\F/\bg})\rightarrow \Ext^{1}(\Omega^{1}_{\F}, \co_{\F})$ and $d\eta_{\ch}$ are surjective; we then observe that
\begin{equation*}
\dim \Ext^{1}(\Omega^{1}_{\F}, \co_{\F}) = \dim \Ext^{1}(\Omega^{1}_{\X}, \co_{\X}),
\end{equation*}
using a result of Charles \cite{Cha} which relates the automorphism group of $\F$ to the one of $\X$. 
The condition $\H^{0}(\Theta_{\X})=0$, which holds for example for Lefschetz cubics, then guarantees that both $\cd_{\X}$ and $\cd_{\F}$ are pro-representable. 
Without assuming $\H^{0}(\Theta_{\X})=0$, we show that $\eta_{\ch}$ is an isomorphism and $\eta$ is surjective. 

We should discuss the relation of our functorial approach to the work of Borcea \cite{Bor} and Wehler \cite{Weh}. 
Writing $\X=\Z(f)$ for $f\in\H^{0}(\co_{\bp}(3))$, Borcea \cite{Bor} considers the deformation of $\F$ given by varying $f$ in $\H^{0}(\co_{\bp}(3))$. 
He checks the conditions
\begin{equation}\label{eq:Wehler}
\H^{1}(\S^{3}\cs^{\vee}\otimes\ci_{\F/\bg})=0 \quad \mathrm{and} \quad \H^{1}(\Theta_{\bg}\vert_{\F})=0,
\end{equation}
which guarantee the completeness of the deformation \cite{Bor, Weh}, for $d\geqslant 6$. 
In contrast to his and other papers \cite{AlK, DeM} using similar methods, 
we explicitly compute the decomposition of the sheaves $\Lambda^{n}\S^{3}\cs$ (which occur in the Koszul resolution of $\ci_{\F/\bg}$) into Schur powers. 
This allows us to check the conditions (\ref{eq:Wehler}), which play an important role in our proof, for all $d\geqslant 5$, 
thus extending Borcea's result to $d=5$. In \Cref{hilbert}, we use the decomposition of $\Lambda^{n}\S^{3}\cs$ 
to express the Hilbert polynomial $\chi(\co_{\F}(n))$ of $\F$ in terms of the Pochhammer symbol; this generalises previous results of Altman and Kleiman \cite{AlK} and Libgober \cite{Lib}.

\section{Auxiliary results}

\subsection{The Borel-Bott-Weil theorem}

Let $\V$ be a complex vector space of dimension $d+2$. We write $\bp=\bp(\V)$ for the projective space of one-dimensional linear subspaces of $\V$, 
and $\bg=\mathrm{Gr}(2,\V)$ for the Grassmannian of lines. On $\bg$ there is a universal exact sequence
\begin{equation}\label{eq:universal}
0\rightarrow\cs\rightarrow\co\otimes\V\rightarrow\cq\rightarrow 0
\end{equation}
of locally free sheaves. The Borel-Bott-Weil theorem \cite[\SS 10]{Bot}, which we will use frequently in this paper, computes the cohomology of sheaves of the form
\begin{equation}\label{eq:bott}
\Sigma^{\lambda}\cq\otimes\Sigma^{\mu}\cs,
\end{equation}
where $\lambda\in\bz^{d}$ and $\mu\in\bz^{2}$ are non-increasing. Here $\Sigma^{\lambda}$ denotes the Schur power corresponding to $\lambda$, 
generalizing the symmetric power $\Sigma^{(k)}=\S^{k}$ and the exterior power $\Sigma^{(1^{k})}=\Lambda^{k}$. 

\begin{theorem}[Borel-Bott-Weil]
Let $\nu=(\lambda,\mu)\in\bz^{d+2}$ and $\rho=(d+2, d+1, \ldots, 1)$. If the components of $\nu+\rho$ are pairwise distinct, then the only
nonvanishing cohomology group of the sheaf $\Sigma^{\lambda}\cq\otimes\Sigma^{\mu}\cs$ is
\begin{equation*}
\H^{l(\sigma)}(\Sigma^{\lambda}\cq\otimes\Sigma^{\mu}\cs)=\Sigma^{\sigma(\nu+\rho)-\rho}\V,
\end{equation*}
where $\sigma\in\mathfrak{S}_{d+2}$ is the unique permutation such that $\sigma(\nu+\rho)$ is non-increasing, and $l(\sigma)$ is its length. If the components of $\nu+\rho$ are not pairwise distinct, then $\H^{\ast}(\Sigma^{\lambda}\cq\otimes\Sigma^{\mu}\cs)=0$.
\end{theorem}

We will in particular rely on the following standard applications, where we tacitly use the canonical isomorphism $\S^{\vee}=\cs\otimes\det(\cq)$.

\begin{example}\label{bbw}
(i) We have $\H^{0}(\S^{n}\cs^{\vee})=\S^{n}\V^{\vee}$ and $\H^{m}(\S^{n}\cs^{\vee})=0$ for $m\geqslant 1$. \\
(ii) Using the decomposition $\SEnd(\cs) = (\Lambda^{2}\cs\oplus\S^{2}\cs)\otimes\det(\cq)$, we obtain
\begin{equation*}
\H^{0}(\SEnd(\cs)) = \bc \quad \mathrm{and} \quad \H^{m}(\SEnd(\cs)) = 0 \quad (m\geqslant 1).
\end{equation*}
(iii) Tensoring (\ref{eq:universal}) with $\cs^{\vee}$, using (ii) and $\Theta_{\bg}=\SHom(\cs, \cq)$, we get
\begin{equation*}
\End(\V)/(1)\xrightarrow{\sim}\H^{0}(\Theta_{\bg}) \quad \mathrm{and} \quad \H^{m}(\Theta_{\bg}) = 0 \quad (m\geqslant 1).
\end{equation*}
\end{example}

\subsection{Fano schemes}

Let $\S$ be a scheme, and $\bp_{\S}=\bp\times\S$. For a closed subscheme $\X\subset\bp_{\S}$, we denote by
\begin{equation*}
\F(\X/\S) = \Hilb^{\T+1}(\X/\S)
\end{equation*}
the relative Hilbert scheme of lines (Fano scheme). Consider the universal subscheme 
\begin{equation*}
\cl_{\S}\subset\bp_{\S}\times_{\S} \F(\bp_{\S}/\S)
\end{equation*}
and write $q_{\S}$ and $p_{\S}$ for the projections of $\bp_{\S}\times_{\S} \F(\bp_{\S}/\S)$ to $\bp_{\S}$ and $\F(\bp_{\S}/\S)$, respectively. 
By \cite[Theorem 2.17]{AlK}, the closed subscheme $\F(\X/\S)\subset\F(\bp_{\S}/\S)$ is the zero scheme of the canonical morphism
\begin{equation}\label{eq:can}
q_{\S}^{\ast}\ci_{\X/\bp_{\S}}\rightarrow\co_{\cl_{\S}}
\end{equation}
of sheaves on $\bp_{\S}\times_{\S} \F(\bp_{\S}/\S)$. Of course,  $\F(\bp_{\S}/\S)=\bg\times\S$, where we view $\bg$ as $\F(\bp/\Spec(\bc))$; 
writing $\pi:\bg\times\S\rightarrow\bg$ for the projection, we have
\begin{equation*}
\cl_{\S}=\bp(\pi^{\ast}\cs) \quad \mathrm{and} \quad \co_{\cl_{\S}}(1)=q_{\S}^{\ast}(\co_{\bp_{\S}}(1))\vert_{\cl_{\S}}.
\end{equation*} 
If $\X=\Z(f)$ for $f\in\H^{0}(\co_{\bp_{\S}}(3))$, then applying $p_{\S\ast}$ to (\ref{eq:can}) induces a section $\sigma_{f}$ of
\begin{equation*}
p_{\S\ast}\co_{\cl_{\S}}(3) = \pi^{\ast}\S^{3}\cs^{\vee},
\end{equation*}
such that, invoking \cite[Proposition 2.3]{AlK},
\begin{equation}\label{eq:zero}
\Z(\sigma_{f}) = \F(\X/\S).
\end{equation}

\begin{remark}\label{sigma}
The map
\begin{equation*}
\sigma:\H^{0}(\co_{\bp_{\S}}(3))\rightarrow\H^{0}(\pi^{\ast}\S^{3}\cs^{\vee})
\end{equation*}
is an isomorphism.
\end{remark}

If $\S=\Spec(\Lambda)$ is affine, we use the abbreviation $\F(\X/\Lambda)=\F(\X/\Spec(\Lambda))$. 
We first consider $\F(\X)=\F(\X/\bc)$ for a cubic $\X=\Z(f)$, $f\in\H^{0}(\co_{\bp}(3))$. 
This scheme is particularly well-behaved when the singular locus of $\X$ is finite, see Corollary 1.4 and Proposition 1.19 of \cite{AlK}:
\begin{theorem}[Altman-Kleiman]\label{AlKtheorem}
Let $\X$ be a cubic with finitely many singularities. The Hilbert scheme $\F=\F(\X)$ is of pure dimension $2d-4$; moreover, $\F$ is reduced for $d\geqslant 4$.
\end{theorem}
As the rank of $\S^{3}\cs^{\vee}$ is $4$, this result in particular implies that the section $\sigma_{f}$ is regular. 
Hence $\F=\Z(\sigma_{f})$ is a local complete intersection, the Koszul complex
\begin{equation}\label{eq:Koszul}
0\rightarrow\Lambda^{4}\S^{3}\cs\rightarrow\Lambda^{3}\S^{3}\cs\rightarrow\Lambda^{2}\S^{3}\cs\rightarrow\S^{3}\cs\rightarrow\ci_{\F/\bg}\rightarrow 0
\end{equation}
is exact, $\sigma_{f}$ induces a canonical isomorphism
\begin{equation*}
\cn_{\F/\bg}\xrightarrow{\sim}\S^{3}\cs\vert_{\F}^{\vee},
\end{equation*}
and the canonical sheaf of $\F$ is given by $\omega_{\F}=\co_{\F}(4-d)$, where $\co_{\F}(1)$ is given by the Plücker embedding. 
The proof of our main theorem relies on the following result.
\begin{lemma}\label{pletyhsm}
We have
\begin{equation*}
\Lambda^{2}\S^{3}\cs = \Sigma^{5,1}\cs\oplus\Sigma^{3^{2}}\cs, \quad \Lambda^{3}\S^{3}\cs = \Sigma^{6, 3}\cs, \quad \Lambda^{4}\S^{3}\cs = \Sigma^{6^{2}}\cs.
\end{equation*} 
\end{lemma}

\begin{proof}
To compute the decomposition of the plethysm $\Lambda^{n}\S^{m}$ into Schur powers, it suffices to compute the corresponding plethysm of Schur functions
\begin{equation*}
s_{1^{n}}\circ s_{m} = \sum_{\lambda} a^{\lambda}_{n, m} s_{\lambda}.
\end{equation*}
Here the sum is taken over all partitions $\lambda$ of $nm$ with at most $n$ parts, 
and the numbers $a^{\lambda}_{n, m}$ can be expressed in terms of generalized Kostka numbers \cite[I \SS 8]{Mac}. 
For small $n$ and $m$, these coefficients are relatively easy to compute; we find
\begin{align*}
s_{1^2}\circ s_{3} &= s_{5,1}+s_{3^2}, \\
s_{1^3}\circ s_{3} &= s_{7,1^2} + s_{6,3} + s_{5,3,1} + s_{3^2}, \\
s_{1^4}\circ s_{3} &= s_{9,1^3} + s_{8,3,1}+s_{7,4,1}+s_{7,3,1^{2}}+s_{6^2}+s_{6,4,2}+s_{6,3^2}+s_{5^2,1^2}+s_{5,3^{2},1}+s_{3^4}.
\end{align*}
It remains to observe that since $\cs$ has rank $2$, we have $\Sigma^{\lambda}\cs=0$ if $\lambda$ has more than two parts. 
(Note that since $\Lambda^{4}\S^{3}\cs = \det(\S^{3}\cs)$, it is easy to show $\Lambda^{4}\S^{3}\cs=\det(\cs)^{\otimes 6}$ directly.)
\end{proof}

\begin{proposition}\label{coh}
Consider the sheaves $\Lambda^{n}\S^{3}\cs\otimes\Theta_{\bg}$ and $\Lambda^{m}\S^{3}\cs\otimes\S^{3}\cs^{\vee}$ on the Grassmannian $\bg$.
For $d\geqslant 6$, $1\leqslant n\leqslant 4$, and $2\leqslant m\leqslant 4$ the cohomology of these sheaves is zero. 
For $d=5$, the only non-vanishing cohomology groups of these sheaves are
\begin{equation*}
\H^{4}(\Lambda^{2}\S^{3}\cs\otimes\Theta_{\bg})=\V^{\ast} \quad and \quad \H^{5}(\Lambda^{2}\S^{3}\cs\otimes\S^{3}\cs^{\vee})=\V^{\ast}.
\end{equation*}
For any $d\geqslant 5$, the only non-vanshing cohomology group of $\S^{3}\cs\otimes\S^{3}\cs^{\vee}$ is
\begin{equation*}
\H^{0}(\S^{3}\cs\otimes\S^{3}\cs^{\vee})=\det(\V)^{\otimes 3}.
\end{equation*}
\end{proposition}

\begin{proof}
	By applying \Cref{pletyhsm}, $\cs^{\vee}=\cs\otimes\det(\cq)$, and the Pieri rule, we obtain
	\begin{align*}
		\S^{3}\cs\otimes\Theta_{\bg} &= \Sigma^{2,1^{d-1}}\cq\otimes(\S^{4}\cs\oplus\Sigma^{3,1}\cs)\\
		\Lambda^{2}\S^{3}\cs\otimes\Theta_{\bg} &= \cq\otimes(\S^{5}\cs\oplus\Sigma^{4,1}\cs\oplus\Sigma^{3,2}\cs), \\
		\Lambda^{3}\S^{3}\cs\otimes\Theta_{\bg} &= \cq\otimes(\Sigma^{6,2}\cs\oplus\Sigma^{5,3}\cs), \\
		\Lambda^{4}\S^{3}\cs\otimes\Theta_{\bg} &= \cq\otimes\Sigma^{6,5}\cs,
	\end{align*}
	for all $d\geqslant 3$. Similarly, we have the decompositions
	\begin{align*}
		S^{3}\cs\otimes\S^{3}\cs^{\vee} &= \Sigma^{3^d}\cq\otimes(\S^{6}\cs\oplus\Sigma^{5,1}\cs\oplus\Sigma^{4,2}\cs\oplus\Sigma^{3,3}\cs),\\
		\Lambda^{2}\S^{3}\cs\otimes\S^{3}\cs^{\vee} &= \Sigma^{3^d}\cq\otimes (\Sigma^{8,1}\cs\oplus\Sigma^{7,2}\cs\oplus\Sigma^{6,3}\cs^{\oplus 2}\oplus\Sigma^{5,4}\cs), \\
		\Lambda^{3}\S^{3}\cs\otimes\S^{3}\cs^{\vee} &= \S^{6}\cs\oplus\Sigma^{5,1}\cs\oplus\Sigma^{4,2}\cs\oplus\Sigma^{3,3}\cs, \\
		\Lambda^{4}\S^{3}\cs\otimes\S^{3}\cs^{\vee} &= \Sigma^{6,3}\cs.
	\end{align*}
	for all $d\geqslant 3$. It remains to apply the Borel-Bott-Weil theorem.
\end{proof}
The Koszul resolution (\ref{eq:Koszul}) induces a hypercohomology spectral sequence
\begin{equation*}
\E^{pq}_{1}=\H^{q}(\Lambda^{-p+1}\S^{3}\cs\otimes\cf)\Rightarrow\H^{p+q}(\ci_{\F/\bg}\otimes\cf)
\end{equation*}
for any locally free sheaf $\cf$ on $\bg$.

\begin{corollary}\label{degen}
For $d\geqslant 5$, the hypercohomology spectral sequences
\begin{gather*}
\E^{pq}_{1}=\H^{q}(\Lambda^{-p+1}\S^{3}\cs\otimes\Theta_{\bg})\Rightarrow\H^{p+q}(\ci_{\F/\bg}\otimes\Theta_{\bg}) \\ \E^{pq}_{1}=\H^{q}(\Lambda^{-p+1}\S^{3}\cs\otimes\S^{3}\cs^{\vee})\Rightarrow\H^{p+q}(\ci_{\F/\bg}\otimes\S^{3}\cs^{\vee})
\end{gather*}
degenerate at the $\E_{1}$-page.
\end{corollary}
For certain classes of complete intersections (including cubics of dimension $d\geqslant 6$), the latter result was obtained by Borcea \cite[\SS 5]{Bor}; 
our approach is similar to his, but Borcea does not explicitly compute the plethysms of \Cref{pletyhsm} --- by employing weight considerations, 
he instead proves a vanishing theorem (which, by \Cref{coh}, does not hold for $d=5$).

\subsection{Deformation theory}

We recall now some well-known general facts about functors of Artin rings, and explain our notation; 
for us an Artin ring is a local $\bc$-algebra which is finite over $\bc$. For a functor of Artin rings $\cf$, 
we denote by $t_{\cf}=\cf(\bc[\varepsilon])$ the tangent space of $\cf$, and if $\varphi:\cf\rightarrow\cg$ is a functorial morphism, 
we refer to
\begin{equation*}
d\varphi=\varphi(\bc[\varepsilon]):t_{\cf}\rightarrow t_{\cg}
\end{equation*}
as the differential of $\varphi$. For future reference, we state the following result \cite[Remark 2.3.8]{Ser}:
\begin{lemma}\label{artinfunctors}
Let $\varphi:\cf\rightarrow\cg$ be a morphism of functors of Artin rings. \\
(i) If $\cf$ and $\cg$ have a pro-representable hull, $\cf$ is smooth and $d\varphi$ surjective, then $\varphi$ is smooth. \\
(ii) If $\cf$ and $\cg$ are pro-representable, $\cf$ is smooth and $d\varphi$ bijective, then $\varphi$ is an isomorphism.
\end{lemma}

The local moduli functor $\cd_{\S}$ of a projective scheme $\S$ takes an Artin ring $\Lambda$ to the set $\cd_{\S}(\Lambda)$ 
of isomorphism classes of deformations of $\S$ over $\Lambda$. It has three basic properties:
\begin{theorem}\label{moduli}
(i) The functor $\cd_{\S}$ has a pro-representable hull. \\
(ii) If $\H^{0}(\Theta_{\S})=0$, then $\cd_{\S}$ is pro-representable. \\
(iii) If $\S$ is reduced, then $t_{\cd_{\S}}=\Ext^{1}(\Omega^{1}_{\S}, \co_{\S})$; if $\S$ is also a local complete intersection, 
then $\Ext^{2}(\Omega^{1}_{\S}, \co_{\S})$ is an obstruction space for $\cd_{\S}$.
\end{theorem}
We refer to Theorem 2.4.1, Proposition 2.4.8, and Corollary 2.6.4 of \cite{Ser}; 
(i) is originally due to Schlessinger \cite[Proposition 3.10]{Schle}, see \cite[Proposition 4]{Gro} for (ii). 
For a closed subscheme $\Z$ of $\S$, the Hilbert functor of $\S$ induces a functor of Artin rings $\ch_{\Z/\S}$ (the local Hilbert functor), 
which takes an Artin ring $\Lambda$ to the set $\ch_{\Z/\S}(\Lambda)$ of deformations of $\Z$ in $\S$ over $\Lambda$. 
By the existence of the Hilbert scheme of $\S$, $\ch_{\Z/\S}$ is pro-representable and $t_{\ch_{\Z/\S}}=\H^{0}(\cn_{\Z/\S})$. 
It is related to $\cd_{\Z}$ by a forgetful morphism
\begin{equation*}
\ch_{\Z/\S}\rightarrow\cd_{\Z}.
\end{equation*}
In view of the description (\ref{eq:zero}) of the Fano scheme as a zero scheme, we will need some specific results about deformations of zero schemes of sections.
\begin{lemma}\label{flatness}
Consider a local scheme $\Spec(\Lambda)$, a locally free sheaf $\cf$ on $\S\times\Spec(\Lambda)$, and a section $\sigma\in\H^{0}(\cf)$. 
If $\sigma\vert_{\S}\in\H^{0}(\cf\vert_{\S})$ is regular, then $\Z(\sigma)\rightarrow\Spec(\Lambda)$ is flat.
\end{lemma}
In particular, the morphism $\Z(\sigma)\rightarrow\Spec(\Lambda)$ is a deformation of $\Z(\sigma\vert_{\S})$ in $\S$ over $\Lambda$. 
After taking an open affine cover of $\S\times\Spec(\Lambda)$ trivializing $\cf$, this is a consequence of the equational criterion for flatness \cite[Example A.12]{Ser}. 
For the rest of this section, we assume that $\S$ is smooth, and $\cf$ is a locally free sheaf on $\S$.

\begin{lemma}\label{basic}
Let $\Z=\Z(\sigma)$ be the zero scheme of a regular section $\sigma\in\H^{0}(\cf)$. \\
(i) The differential of $\ch_{\Z/\S}\rightarrow\cd_{\Z}$ can be identified with the connecting morphism
\begin{equation*}
\H^{0}(\cn_{\Z/\S}) \rightarrow \Ext^{1}(\Omega^{1}_{\Z}, \co_{\Z})
\end{equation*}
associated to conormal sequence of $\Z\subset\S$. \\
(ii) Under the canonical identification $\cf\vert_{\Z}\simeq \cn_{\Z/\S}$, the restriction map
\begin{equation*}
\H^{0}(\cf)\rightarrow\H^{0}(\cn_{\Z/\S})
\end{equation*}
takes $\tau\in\H^{0}(\cf)$ to the first-order deformation of $\Z$ in $\S$ given by $\Z(\sigma+\varepsilon \tau)$.
\end{lemma}
We refer to \cite[Remark 3.2.10]{Ser} for (i); (ii) is a consequence of the standard identification $\ch_{\Z/\S}(\bc[\varepsilon])=\H^{0}(\cn_{\Z/\S})$ 
(see for instance the proof of \cite[Proposition 3.2.1]{Ser}), and \Cref{flatness} for $\Lambda=\bc[\varepsilon]$.

Finally, we consider the projection $\varphi:\S\times\H^{0}(\cf)\rightarrow\S$. 
There is a tautological section $\zeta\in\H^{0}(\varphi^{\ast}\cf)$ such that $\zeta(s, \sigma)=\sigma(s)$ for every point $(s, \sigma)$ of $\S\times\H^{0}(\cf)$; 
in particular, $\zeta\vert_{\S\times\{\sigma\}}=\sigma$. Let $\cz=\Z(\zeta)$ be its zero scheme, and 
\begin{equation*}
\pi:\cz\rightarrow\H^{0}(\cf)
\end{equation*}
be the projection. Then $\pi^{-1}(\sigma)=\Z(\sigma)$, and if $\sigma$ is regular, then \Cref{flatness} implies that $\pi$ is flat in a neighbourhood of $\sigma$, 
thus inducing a deformation of $\Z=\Z(\sigma)$. The following result \cite[Theorem 1.5]{Weh} gives a criterion for the completeness of the latter deformation.
\begin{lemma}[Wehler]\label{Wehl}
If $\H^{1}(\cf\otimes\ci_{\Z/\S})=0$ and $\H^{1}(\Theta_{\S}\vert_{\Z})=0$, then the Kodaira-Spencer map
\begin{equation*}
\kappa_{\pi, \sigma}: \H^{0}(\cf)\rightarrow\Ext^{1}(\Omega^{1}_{\Z}, \co_{\Z})
\end{equation*}
is surjective.
\end{lemma}

\section{The Hilbert polynomial of $\F$}

\subsection{Related results} Using Schubert calculus, Altman and Kleiman \cite[Proposition 1.6]{AlK} prove that the Plücker degree of $\F$ is given by
\begin{equation*}
\int_{\F} c_{1}(\co_{\F}(1))^{2d-4} = 27\frac{(2d-4)!}{d!(d-1)!}(3d^{2}-7d+4).
\end{equation*}
In the special case $d=3$, this is a theorem of Fano \cite[\SS 2]{Fan}. It is thus a natural question to determine, more generally, the Hilbert polynomial
\begin{equation*}
\chi(\co_{\F}(n)) = \sum_{k=0}^{2d-4} \frac{n^{k}}{k!}\int_{\F} c_{1}(\co_{\F}(1))^{k}\cap\mathrm{Td}(\F).
\end{equation*}
Altman and Kleiman (and, independently, Libgober \cite[\SS 2]{Lib}) show that for $d=3$, we have
\begin{equation*}
\chi(\co_{\F}(n)) = \frac{45}{2}n^{2}-\frac{45}{2}n+6.
\end{equation*}
In this section, we use \Cref{pletyhsm} to express the Hilbert polynomial $\chi(\co_{\F}(n))$, for any dimension $d$, in terms of the Pochhammer symbol.

\subsection{$\chi(\co_{\F}(n))$ via the Pochhammer symbol}

Recall that the Pochhammer symbol $(x)_{d}$ is defined by
\begin{equation*}
(x)_{d} = \prod_{j=0}^{d-1}(x+j).
\end{equation*}
	
\begin{theorem}\label{hilbert}
The Hilbert polynomial of $\F$ is given by
\begin{align*}
\chi(\co_{\F}(n)) &= \frac{1}{d!(d+1)!}\left(  (n+1)_{d} (n+2)_{d} - 4(n-2)_{d} (n+2)_{d} + (n-2)_{d} (n-1)_{d} \right.\\
&\left.+ 5(n-4)_{d} (n+1)_{d} - 4 (n-5)_{d} (n-1)_{d} + (n-5)_{d} (n-4)_{d}\right) .
\end{align*}
\end{theorem}

\begin{proof}
By the Koszul resolution, we obtain
\begin{equation*}
\chi(\co_{\F}(n)) = \chi(\co_{\G}(n)) - \chi(\S^{3}\cs(n)) + \chi(\Lambda^{2}\S^{3}\cs(n)) - \chi(\Lambda^{3}\S^{3}\cs(n)) + \chi(\Lambda^{4}\S^{3}\cs(n)).
\end{equation*}
Using \Cref{pletyhsm}, it suffices to describe the Hilbert polynomial of $\Sigma^{\mu_{1}, \mu_{2}}\cs$ for any $\mu_{1}\geqslant\mu_{2}$. We now establish the equality
 \begin{equation}\label{eq:poly}
\chi(\Sigma^{\mu_{1}, \mu_{2}}\cs(n)) = \frac{(\mu_{1}-\mu_{2}+1)}{(d+1)!d!}(n-\mu_{1}+1)_{d}(n-\mu_{2}+2)_{d}.
\end{equation}
To prove (\ref{eq:poly}), we may assume $n\geqslant\mu_{1}$. Since $\co_{\G}(n)=\Sigma^{(n^{d})}\cq$, 
\begin{equation*}
\chi(\Sigma^{n^{d}}\cq\otimes\Sigma^{\mu_{1}, \mu_{2}}\cs)=\dim\Sigma^{n^{d}, \mu_{1}, \mu_{2}}\V
\end{equation*}
by the Borel-Bott-Weil theorem. Hence
\begin{equation*}
\dim\Sigma^{n^{d}, \mu_{1}, \mu_{2}}\V = (\mu_{1}-\mu_{2}+1)\prod_{j=1}^{d} \frac{(n-\mu_{1}+j)(n-\mu_{2}+j+1)}{j(j+1)}
\end{equation*}
by the Weyl dimension formula.
In particular, combining (\ref{eq:poly}) with \Cref{pletyhsm}, we obtain
\begin{gather*}
\chi(\co_{\G}(n)) =  \frac{1}{d!(d+1)!} (n+1)_{d} (n+2)_{d}, \quad
\chi(\S^{3}\cs(n)) =  \frac{4}{d!(d+1)!} (n-2)_{d} (n+2)_{d}, \\
\chi(\Lambda^{2}\S^{3}\cs(n)) = \frac{1}{d!(d+1)!} (n-2)_{d} (n-1)_{d} + \frac{5}{d!(d+1)!} (n-4)_{d} (n+1)_{d}, \\
\chi(\Lambda^{3}\S^{3}\cs(n)) = \frac{4}{d!(d+1)!} (n-5)_{d} (n-1)_{d}, \quad
\chi(\Lambda^{4}\S^{3}\cs(n)) = \frac{1}{d!(d+1)!} (n-5)_{d} (n-4)_{d}.
\end{gather*}
\end{proof}

A result of Schlömlich \cite[\SS 3]{Schl} explicitly describes the coefficients of
\begin{equation*}
(x)_{d}= \sum_{k=0}^{d} {d\brack k}x^{k},
\end{equation*}
which are the (unsigned) Stirling numbers of the first kind, in terms of binomial coefficients:
\begin{equation*}
{d\brack k} = (-1)^{d-k}\sum_{m=0}^{d-k} \binom{d-1+m}{k-1}\binom{2d-k}{d+m}\sum_{n=0}^{m} \frac{(-1)^{n} n^{d-k+m}}{n!(m-n)!}.
\end{equation*}
Combining this with \Cref{hilbert}, we can express the coefficients of $\chi(\co_{\F}(n))$ as sums of binomial coefficients; 
this shows in particular that our expression for $\chi(\co_{\F}(n))$ is a polynomial of degree $2d-4$, as it should be.

\begin{corollary}
We have the expansion
\begin{equation*}
\chi(\co_{\F}(n)) = 27\frac{(3d-4)}{d!(d-2)!}n^{2d-4}+27\frac{(3d-4)(d-4)}{d!(d-3)!}n^{2d-5}+\cdots.
\end{equation*}
\end{corollary}

\Cref{hilbert} and Kodaira vanishing 
\begin{equation*}
\chi(\co_{\F}(n))=h^{0}(\co_{\F}(n)) \quad (n\geqslant 5-d)
\end{equation*}
allow one to compute the dimension of the space $\H^{0}(\ci_{\F/\bg}(n))$ of global sections of $\co_{\bg}(n)$ vanishing on $\F$. 
Indeed, for $d\geqslant 4$, Debarre and Manivel \cite[Theorem 4.1]{DeM} prove that $\H^{1}(\ci_{\F/\bg}(n))=0$ for $n\geqslant 0$. 
There is thus an exact sequence of the form
\begin{equation*}
0\rightarrow\H^{0}(\ci_{\F/\bg}(n))\rightarrow\H^{0}(\co_{\bg}(n))\rightarrow\H^{0}(\co_{\F}(n))\rightarrow 0.
\end{equation*}

\subsection{Examples} 

Writing $\X_{d}$ to indicate the dimension $d$ of the cubic $\X$, we have
\begin{gather*}
\chi(\co_{\F(\X_{4})}(n)) = \frac{9}{2}n^{4} + \frac{15}{2}n^{2} +3, \\
\chi(\co_{\F(\X_{5})}(n)) = \frac{33}{80}n^{6} + \frac{99}{80}n^{5} + \frac{57}{16} n^{4} + \frac{81}{16}n^{3} + \frac{241}{40}n^{2} + \frac{37}{10}n + 1, \\
\chi(\co_{\F(\X_{6})}(n)) =\frac{7}{320}n^{8} + \frac{7}{40}n^{7} + \frac{391}{480}n^{6} + \frac{39}{16}n^{5} + \frac{4889}{960}n^{4} + \frac{591}{80}n^{3} + \frac{1697}{240}n^{2} +4n +1,  \\
\chi(\co_{\F(\X_{7})}(n)) = \frac{17}{22400}n^{10} +\frac{51}{4480}n^{9}+\frac{589}{6720}n^{8}+\frac{979}{2240}n^{7}+\frac{4903}{3200}n^{6}+\frac{2493}{640}n^{5} \\
+\frac{4023}{560}n^{4}+\frac{10503}{1120}n^{3}+\frac{34421}{4200}n^{2}+\frac{599}{140}n+1.
\end{gather*}

\section{Deformations of $\X$}

\subsection{Generalities} 

Consider a cubic $\X\subset\bp$ of dimension $d\geqslant 3$, having only finitely many singularities, and defined by $f\in\H^{0}(\co_{\bp}(3))$.
\begin{lemma}\label{cubic}
(i) The restriction map
\begin{equation*}
\H^{0}(\co_{\bp}(3))\rightarrow\H^{0}(\co_{\X}(3))
\end{equation*}
is surjective with kernel $(f)$. \\
(ii) The restriction map
\begin{equation*}
\H^{0}(\Theta_{\bp})\rightarrow\H^{0}(\Theta_{\bp}\vert_{\X})
\end{equation*}
is an isomorphism, and $\H^{1}(\Theta_{\bp}\vert_{\X})=0$. \\
(iii) We have $\Ext^{2}(\Omega^{1}_{\X}, \co_{\X})=0$.
\end{lemma}
The proof is straightforward. Part (iii) implies that $\cd_{\X}$ is smooth, and a consequence of (ii) is that the forgetful morphism
\begin{equation*}
\ch_{\X/\bp}\rightarrow\cd_{\X}
\end{equation*}
is smooth, in particular surjective. 

\begin{remark}\label{cubicrem}
In fact, any deformation $\fx\subset\bp_{\Lambda}$ of $\X$ in $\bp$ over an Artin ring $\Lambda$ is a cubic: 
there exists a section $f_{\Lambda}$ of $\co_{\bp_{\Lambda}}(3)$ extending $f$, such that $\fx=\Z(f_{\Lambda})$ \cite[Theorem 1]{Weh2}.
\end{remark}

\subsection{Automorphisms}

As the vanishing of $\H^{0}(\Theta_{\X})$ guarantees the pro-representability of $\cd_{\X}$, we are led to study $\H^{0}(\Theta_{\X})$. 
It is well-known that $\H^{0}(\Theta_{\X})=0$ when $\X$ is smooth (see \cite[\SS 5]{Jor}, \cite[Lemma 14.2]{KoS}). We now extend this result to the class of Lefschetz cubics
in the sense of \cite[Definition 5.1]{ClG}; apart from smooth cubics, this class consists of the simplest singular cubics: those whose singular
locus consists of a single node. Observe that $\H^{0}(\Theta_{\X})$ is the kernel of the derivative
\begin{equation*}
df:\H^{0}(\Theta_{\bp}\vert_{\X})\rightarrow\H^{0}(\co_{\X}(3)),
\end{equation*}
which under the identification
\begin{equation}
\begin{tikzcd}
\H^{0}(\Theta_{\bp}\vert_{\X}) \arrow[r, "df"] & \H^{0}(\co_{\X}(3)) \\
\H^{0}(\Theta_{\bp}) \arrow[u, "\isom"] \arrow[r] & \H^{0}(\co_{\bp}(3))/(f) \arrow[u, swap, "\isom"] 	
\end{tikzcd}
\end{equation}	
is given by $df(\sum L_{i}\partial_{i})=\sum L_{i}\partial_{i} f \mod (f)$. 
We can thus view $\H^{0}(\Theta_{\X})$ as the subspace of $\H^{0}(\Theta_{\bp})$ consisting of all $\sum L_{i}\partial_{i}$ such that 
\begin{equation}\label{eq:euler}
\sum L_{i}\partial_{i} f = \lambda f 
\end{equation}
for some constant $\lambda$.

\begin{proposition}\label{lefschetz}
If $\X$ is a Lefschetz cubic, then $\H^{0}(\Theta_{\X})=0$.
\end{proposition}

\begin{proof}
Consider a cubic $\X$ with a single node $x_{0}$. After a linear change of coordinates, we may assume $x_{0}=[0:\cdots:0:1]$.
Then the equation defining $\X$ can be written as
\begin{equation}\label{eq:split}
f(x_{0},\ldots,x_{d+1})=g(x_{0},\ldots,x_{d})+x_{d+1}h(x_{0},\ldots,x_{d}),
\end{equation}
where $g$ is a cubic and $h$ a non-degenerate quadric. Inserting (\ref{eq:split}) into (\ref{eq:euler}), we have to show that if
\begin{equation}\label{eq:derivsplit}
\sum_{i=0}^{d} L_{i}\partial_{i}g + x_{d+1}\sum_{i=0}^{d} L_{i}\partial_{i}h + L_{d+1}h = \lambda(g+x_{d+1}h)
\end{equation}
for some constant $\lambda$, then $L_{i}=\mu x_{i}$ for some constant $\mu$. Write
\begin{equation*}
L_{i}(x_{0},\ldots, x_{d+1}) = \lambda_{i}x_{d+1}+l_{i}(x_{0}, \ldots, x_{d}).
\end{equation*}
Taking the coefficient of $x_{d+1}^{2}$ in (\ref{eq:derivsplit}), we obtain
\begin{equation*}
\sum_{i=0}^{d} \lambda_{i}\partial_{i}h=0,
\end{equation*}
and in particular, since $h$ is non-degenerate, $\lambda_{i}=0$ for $0\leqslant i\leqslant d$. 
On the other hand, taking the coefficient of $x_{d+1}$ in (\ref{eq:derivsplit}) gives
\begin{equation}\label{eq:derivsplit2}
\sum_{i=0}^{d} l_{i}\partial_{i} h + \lambda_{d+1}h = \lambda h, \quad \sum_{i=0}^{d} l_{i}\partial_{i} g + l_{d+1}h = \lambda g.
\end{equation}
Consider now the linear subspace $\bp^{'}=\Z(x_{d+1})\subset\bp$, and the smooth complete intersection $\Z=\Z(g,h)\subset\bp^{'}$. 
The restriction maps induce isomorphisms
\begin{gather*}
\H^{0}(\Theta_{\bp^{'}}) \xrightarrow{\sim} \H^{0}(\Theta_{\bp^{'}}\vert_{\Z}), \quad
\H^{0}(\co_{\bp^{'}}(2))/(h) \xrightarrow{\sim} \H^{0}(\co_{\Z}(2)), \\
\H^{0}(\co_{\bp^{'}}(3))/(g, h\H^{0}(\co_{\bp^{'}}(1))) \xrightarrow{\sim} \H^{0}(\co_{\Z}(3)).
\end{gather*}
Using these isomorphisms, one can, parallel to our description of $\H^{0}(\Theta_{\X})$, 
explicitly describe $\H^{0}(\Theta_{\Z})$ as a subspace of $\H^{0}(\Theta_{\bp^{'}})$. 
Then (\ref{eq:derivsplit2}) precisely means that
\begin{equation*}
\sum_{i=0}^{d} l_{i}\partial_{i}\in\H^{0}(\Theta_{\Z}).
\end{equation*}
Since $\Z$ is smooth, we have $\H^{0}(\Theta_{\Z})=0$; in particular, $l_{i}(x_{0}, \ldots, x_{d})=\mu x_{i}$ for a constant $\mu$. 
Inserting this into (\ref{eq:derivsplit2}) gives 
\begin{equation*}
2\mu h+\lambda_{d+1}h = \lambda h, \quad 3\mu g+l_{d+1}h = \lambda g.
\end{equation*}
Since $\X$ is irreducible, the second equation implies $l_{d+1}=0$ and $\lambda=3\mu$, while the first one gives $\lambda_{d+1}=\lambda-2\mu=\mu$.
\end{proof}

More generally, we expect that
\begin{equation*}
\H^{0}(\Theta_{\X})=0
\end{equation*}
for any nodal cubic.  Low-dimensional ($2\leqslant d\leqslant 4$) nodal cubics are in fact known to be stable 
in the sense of geometric invariant theory \cite[Theorem 1.1]{Laz}, and so the vanishing of $\H^{0}(\Theta_{\X})$ holds for $2\leqslant d\leqslant 4$.

\subsection{Locally trivial deformations}

Instead of $\cd_{\X}$, one could consider the subfunctor $\cd_{\X}^{'}$ of $\cd_{\X}$ given by the locally trivial deformations of $\X$; 
here $\H^{2}(\Theta_{\X})$ is an obstruction space of $\cd_{\X}^{'}$ \cite[Theorem 2.4.1]{Ser}. 
While it is known that if $d=2$ or $d=3$, then $\H^{2}(\Theta_{\X})=0$ \cite[Proposition 4]{Nam}, this vanishing need not hold when $d$ is large. 
In fact, the following holds:
\begin{proposition}\label{manynodes}
Let $\X$ be a nodal cubic with $\H^{0}(\Theta_{\X})=0$. If $\X$ has $\delta>\binom{d+2}{3}$ nodes, then
\begin{equation*}
\H^{2}(\Theta_{\X}) \neq 0.
\end{equation*} 
\end{proposition}
Indeed, $\H^{0}(\Theta_{\X})=0$ and \Cref{cubic} imply that
\begin{equation*}
\dim \Ext^{1}(\Omega^{1}_{\X}, \co_{\X}) = \binom{d+2}{3}, 
\end{equation*} 
and there is an exact sequence of the form
\begin{equation*}
\Ext^{1}(\Omega^{1}_{\X}, \co_{\X})\rightarrow\H^{0}(\SExt^{1}(\Omega^{1}_{\X}, \co_{\X}))\rightarrow\H^{2}(\Theta_{\X})\rightarrow 0,
\end{equation*}
coming from the local-to-global spectral sequence; here $\SExt^{1}(\Omega^{1}_{\X}, \co_{\X})$ is the structure sheaf of the singular locus. 
As a special case of a result of Varchenko \cite[\SS 2]{Var}, 
\begin{equation*}
\delta\leqslant \binom{d+2}{[\frac{d+1}{2}]}
\end{equation*}
which turns out to be optimal; hence $\delta>\binom{d+2}{3}$ is possible only for $d\geqslant 7$.

\begin{remark}
The space $\H^{2}(\Theta_{\X})$ is canonically isomorphic to $\H^{1}(\cn^{'}_{\X/\bp})$, 
where $\cn^{'}_{\X/\bp}$ is the equisingular normal sheaf of $\X\subset\bp$. 
We can view $\X$ as a point of the Hilbert scheme $\V^{\delta}_{d}$ of cubic hypersurfaces in $\bp$ with $\delta$ nodes (Severi scheme); 
$\H^{0}(\cn^{'}_{\X/\bp})$ and $\H^{1}(\cn^{'}_{\X/\bp})$ are then the tangent and obstruction spaces of $\V^{\delta}_{d}$ at $[\X]$ \cite[\SS 3]{GrK}. 
\Cref{manynodes} naturally leads to an extension of Theorem 111 of \cite{Cat}.
\end{remark}

\section{Deformations of $\F$.}

\subsection{The functorial morphism $\eta$.}

Consider a cubic $\X$ with finitely many singularities, and an infinitesimal deformation $\fx$ of $\X$ over an Artin ring $\Lambda$. 
Then $\fx$ is induced by a deformation $\fx\subset\bp_{\Lambda}$ of $\X$ in $\bp$, and $\fx\subset\bp_{\Lambda}$ is a cubic (\Cref{cubicrem}). 
Using the induced polarisation $\co_{\fx}(1)$ of $\fx$ over $\Lambda$, we can consider the relative Hilbert scheme of lines $\F(\fx/\Lambda)$, 
which is naturally a closed subscheme of $\bg_{\Lambda}$. Recalling the zero scheme description (\ref{eq:zero}) and the regularity of the section defining $\F$, 
\Cref{flatness} implies that the morphism
\begin{equation*}
\F(\fx/\Lambda)\rightarrow\Spec(\Lambda)
\end{equation*}
is flat. In particular, $\F(\fx/\Lambda)$ can be thought of as an infinitesimal deformation of $\F$ in $\bg$ over $\Lambda$. 
For any morphism of local Artin rings $\Lambda\rightarrow\Lambda^{'}$, we have
\begin{equation*}
\F(\mathfrak{X} / \Lambda)\times_{\Lambda} \Lambda^{'} = \F(\fx_{\Lambda^{'}}/\Lambda^{'})
\end{equation*}
as a subscheme of $\bg_{\Lambda^{'}}=\bg_{\Lambda}\times_{\Lambda}\Lambda^{'}$. The relative Hilbert scheme thus defines a morphism
\begin{equation*}
\eta_{\ch}:\ch_{\X/\bp}\rightarrow\ch_{\F/\bg}.
\end{equation*}
of local Hilbert functors. Since $\Pic(\fx)=\bz$ by the Grothendieck-Lefschetz theorem and $\omega_{\fx/\Lambda}=\co_{\fx}(1-d)$, 
the isomorphism class of the deformation $\F(\fx/\Lambda)$ of $\F$ over $\X$ depends only on the isomorphism class of the deformation $\fx$ of $\X$ over $\Lambda$, 
and so we get a morphism
\begin{equation}
\eta:\cd_{\X}\rightarrow\cd_{\F},
\end{equation}
related to $\eta_{\ch}$ by a commutative diagram
\begin{equation}\label{eq:comp}
\begin{tikzcd}
\ch_{\X/\bp}  \arrow[r] \arrow[d, swap, "\eta_{\ch}"] & \cd_{\X}  \arrow[d, "\eta"]  \\
\ch_{\F/\bg} \arrow[r] & \cd_{\F}.
\end{tikzcd}
\end{equation}

The proof of our main theorem requires an analogue of \Cref{cubic} for $\F\subset\bg$.
\begin{lemma}\label{fano}
Let $d\geqslant 5$.
(i) The restriction map
\begin{equation*}
\H^{0}(\S^{3}\cs^{\vee})\rightarrow\H^{0}(\S^{3}\cs\vert_{\F}^{\vee})
\end{equation*}
is surjective with kernel $(\sigma_{f})$. \\
(ii) The restriction map
\begin{equation*}
\H^{0}(\Theta_{\bg})\rightarrow\H^{0}(\Theta_{\bg}\vert_{\F})
\end{equation*}
is an isomorphism, and $\H^{1}(\Theta_{\bg}\vert_{\F})=0$. 
\end{lemma}

\begin{proof}
(i) By \Cref{degen}, the spectral sequence
\begin{equation*}
\E^{pq}_{1}=\H^{q}(\Lambda^{-p+1}\S^{3}\cs\otimes\S^{3}\cs^{\vee})\Rightarrow\H^{p+q}(\ci_{\F/\bg}\otimes\S^{3}\cs^{\vee})
\end{equation*}
degenerates at the $\E_{1}$-page. In particular,
\begin{equation*}
\H^{0}(\ci_{\F/\bg}\otimes\S^{3}\cs^{\vee})\simeq\H^{0}(\S^{3}\cs\otimes\S^{3}\cs^{\vee}) \quad \mathrm{and} \quad \H^{1}(\ci_{\F/\bg}\otimes\S^{3}\cs^{\vee})=0.
\end{equation*}
Here $\H^{0}(\S^{3}\cs\otimes\S^{3}\cs^{\vee})$ is one-dimensional (\Cref{coh}), and it remains to combine this with the exact sequence in cohomology associated to
\begin{equation*}
0\rightarrow\ci_{\F/\bg}\otimes\S^{3}\cs^{\vee}\rightarrow\S^{3}\cs^{\vee}\rightarrow\S^{3}\cs\vert_{\F}^{\vee}\rightarrow 0.
\end{equation*}
(ii) Similarly, by \Cref{degen} the spectral sequence 
\begin{equation*}
\E^{pq}_{1}=\H^{q}(\Lambda^{-p+1}\S^{3}\cs\otimes\Theta_{\bg})\Rightarrow\H^{p+q}(\ci_{\F/\bg}\otimes\Theta_{\bg})
\end{equation*}
degenerates at the $\E_{1}$-page, and we obtain
\begin{equation*}
\H^{0}(\ci_{\F/\bg}\otimes\Theta_{\bg}) = \H^{1}(\ci_{\F/\bg}\otimes\Theta_{\bg}) = \H^{2}(\ci_{\F/\bg}\otimes\Theta_{\bg}) = 0.
\end{equation*}
The result follows from this vanishing, and the exact sequence
\begin{equation*}
0\rightarrow\ci_{\F/\bg}\otimes\Theta_{\bg}\rightarrow\Theta_{\bg}\rightarrow\Theta_{\bg}\vert_{\F}\rightarrow 0. \qedhere
\end{equation*}
\end{proof}

\begin{corollary}
The forgetful morphism $\ch_{\F/\bg}\rightarrow\cd_{\F}$ is smooth.
\end{corollary}

We now apply \Cref{Wehl} to $\cf=\S^{3}\cs^{\vee}$ on $\S=\bg$. Let $\pi:\cz\rightarrow\H^{0}(\S^{3}\cs^{\vee})$ be as in \Cref{Wehl}, and put
$\phi=\sigma^{-1}\circ\pi$, where $\sigma$ is the isomorphism of \Cref{sigma}.
\begin{corollary}
The Kodaira-Spencer map $\kappa_{\phi, f}:\H^{0}(\co_{\bp}(3))\rightarrow\Ext^{1}(\Omega^{1}_{\F}, \co_{\F})$ is surjective.
\end{corollary}
In other words, the deformation of $\F$ induced by $\phi$ is complete at $f\in\H^{0}(\co_{\bp}(3))$. 
For $d\geqslant 6$, this is the content of Theorem 5.3 of \cite{Bor}.

For a point $x_{0}$ of $\X$ we let $\Sigma_{x_{0}}\subset\bg$ be the subscheme parametrising lines $\L$ containing $x_{0}$, 
and define $\F_{x_{0}}=\F\cap\Sigma_{x_{0}}$. We have $\F_{x_{0}}\neq\varnothing$ since $\X$ can be covered by lines;
in fact, 
\begin{equation}\label{eq:fdim}
\dim \F_{x_{0}}\geqslant d-4
\end{equation}
as $\Sigma_{x_{0}}$ is a projective space of dimension $d$ and $\F_{x_{0}}=\Z(\sigma_{f}\vert_{\Sigma_{x_{0}}})$.

\begin{lemma}\label{aut}
For $d\geqslant 5$ there is a canonical isomorphism
\begin{equation*}
\H^{0}(\Theta_{\X})\xrightarrow{\sim}\H^{0}(\Theta_{\F}).
\end{equation*}
\end{lemma}

\begin{proof} 
Consider the canonical morphism of automorphism groups
\begin{equation}\label{eq:aut}
\alpha:\Aut(\X)\rightarrow\Aut(\F).
\end{equation}
If $\phi:\X\rightarrow\X$ is an automorphism, then $\alpha(\phi):\F\rightarrow\F$ satisfies 
\begin{equation*}
\alpha(\phi)([\L])=[\phi(\L)]  
\end{equation*}
for every line $\L\subset\X$. We now show that $\alpha$ is injective; let $\phi$ be in the kernel of $\alpha$. By (\ref{eq:fdim}) $\X$ is covered by lines,
and so it suffices to show that $\phi\vert_{\L}:\L\rightarrow\phi(\L)=\L$ is the identity for every line $\L\subset\X$. 
For a point $x_{0}$ of $\L$, $\phi(x_{0})$ lies in $\phi(\M)=\M$ for every line $\M\subset\X$ containing $x_{0}$. 
Since the space $\F_{x_{0}}$ of such lines $\M$ has dimension $\geqslant 1$ by (\ref{eq:fdim}), this is possible only if $\phi(x_{0})=x_{0}$. 
Proposition 4 of \cite{Cha} shows that the image of (\ref{eq:aut}) is the subgroup 
$\Aut(\F, \co_{\F}(1))$ of automorphisms of $\F$ preserving the Plücker polarization. Since $\H^{1}(\co_{\F})=0$, 
$\H^{0}(\Theta_{\F})$ is the tangent space of $\Aut(\F, \co_{\F}(1))$ at the identity and so the differential of $\alpha$ at the identity gives the desired isomorphism
$\H^{0}(\Theta_{\X})\xrightarrow{\sim}\H^{0}(\Theta_{\F})$.
\end{proof}

\subsection{Proof of the main theorem}

\begin{theorem}\label{main}
Let $d\geqslant 5$. Then the differential 
\begin{equation*}
d\eta:\Ext^{1}(\Omega^{1}_{\X}, \co_{\X})\rightarrow\Ext^{1}(\Omega^{1}_{\F}, \co_{\F})
\end{equation*}
of $\eta$ is an isomorphism. If $\H^{0}(\Theta_{\X})=0$, then $\eta$ is an isomorphism.
\end{theorem}

\begin{proof}
Consider the diagram
\begin{equation*}
\begin{tikzcd}
\H^{0}(\cn_{\X/\bp}) \arrow[r] \arrow[d, swap, "d\eta_{\ch}"] & \Ext^{1}(\Omega^{1}_{\X}, \co_{\X})  \arrow[d, "d\eta"]  \\
\H^{0}(\cn_{\F/\bg}) \arrow[r] & \Ext^{1}(\Omega^{1}_{\F}, \co_{\F}) 
\end{tikzcd}
\end{equation*}
of differentials induced by (\ref{eq:comp}). By \Cref{basic} (i) and \Cref{fano} (ii), the differential
\begin{equation*}
\H^{0}(\cn_{\F/\bg})\rightarrow\Ext^{1}(\Omega^{1}_{\F}, \co_{\F})
\end{equation*} 
of the forgetful morphism is surjective. To show that $d\eta$ is surjective, it remains to observe that $d\eta_{\ch}$ is surjective. The diagram
\begin{equation}\label{eq:aux}
\begin{tikzcd}
\H^{0}(\co_{\bp}(3)) \arrow[r] \arrow[d, swap, "\sigma"] & \H^{0}(\cn_{\X/\bp})  \arrow[d, "d\eta_{\ch}"]  \\
\H^{0}(\S^{3}\cs^{\vee}) \arrow[r] & \H^{0}(\cn_{\F/\bg}),
\end{tikzcd}
\end{equation}
where the horizontal maps are given by restriction, is commutative; indeed, we have
\begin{equation*}
\F(\Z(f+\varepsilon g)/\bc[\varepsilon]) = \Z(\sigma_{f}+\varepsilon\sigma_{g})
\end{equation*}
by (\ref{eq:zero}). Since $\sigma$ is an isomorphism and the restriction map
\begin{equation*}
\H^{0}(\S^{3}\cs^{\vee})\rightarrow \H^{0}(\cn_{\F/\bg})
\end{equation*}
is surjective by \Cref{fano} (i), it follows that $d\eta_{\ch}$ is surjective. It now suffices to show that
\begin{equation*}
\dim \Ext^{1}(\Omega^{1}_{\X}, \co_{\X}) = \dim \Ext^{1}(\Omega^{1}_{\F}, \co_{\F}).
\end{equation*}
Consider the pair of exact sequences
\begin{align*}
0\rightarrow \H^{0}(\Theta_{\X}) \rightarrow \H^{0}(\Theta_{\bp}\vert_{\X}) \rightarrow \H^{0}(\cn_{\X / \bp}) \rightarrow \Ext^{1}(\Omega^{1}_{\X}, \co_{\X}) \rightarrow 0 \\
0\rightarrow \H^{0}(\Theta_{\F}) \rightarrow \H^{0}(\Theta_{\bg}\vert_{\F}) \rightarrow \H^{0}(\cn_{\F / \bg}) \rightarrow \Ext^{1}(\Omega^{1}_{\F}, \co_{\F}) \rightarrow 0
\end{align*}
associated to the conormal sequences of $\X\subset\bp$ and $\F\subset\bg$, respectively. By \Cref{aut} we have $h^{0}(\Theta_{\X}) = h^{0}(\Theta_{\F})$, while
\begin{equation*}
h^{0}(\Theta_{\bp}\vert_{\X}) = h^{0}(\Theta_{\bg}\vert_{\F}), \quad \mathrm{and} \quad h^{0}(\cn_{\X / \bp}) = h^{0}(\cn_{\F / \bg})
\end{equation*}
result from \Cref{cubic}, \Cref{fano}, and \Cref{bbw}.

If $\H^{0}(\Theta_{\X})=0$, then $\H^{0}(\Theta_{\F})=0$ by \Cref{aut}. Hence both $\cd_{\X}$ and $\cd_{\F}$ are pro-representable; 
since $\cd_{\X}$ is smooth and $d\eta$ bijective, it remains to apply \Cref{artinfunctors} (ii).
\end{proof}

\begin{corollary}
The morphism $\eta_{\ch}$ is an isomorphism, and $\eta$ is surjective.
\end{corollary}
\begin{proof}
This is a consequence of the proof of \Cref{main} rather than \Cref{main} itself. The proof shows that $d\eta_{\ch}$ can be identified with the isomorphism 
\begin{equation*}
\H^{0}(\co_{\bp}(3))/(f) \xrightarrow{\sim} \H^{0}(\S^{3}\cs^{\vee})/(\sigma_{f})
\end{equation*}
induced by $\sigma$. As $\ch_{\X/\bp}$ and $\ch_{\F/\bg}$ are pro-representable, and $\ch_{\X/\bp}$ smooth, $\eta_{\ch}$ is an isomorphism by \Cref{artinfunctors} (ii). 
Finally, $\eta$ is surjective by \Cref{artinfunctors} (i), as both $\cd_{\X}$ and $\cd_{\F}$ have a pro-representable hull by Schlessinger's theorem, 
\Cref{moduli} (i).
\end{proof}

\begin{remark}
Through the proof of \Cref{aut}, our proof of \Cref{main} depends on the results of \cite{Cha}. 
One could get rid of this dependence by establishing a commutative diagram
\begin{equation*}\label{eq:aux}
\begin{tikzcd}
\H^{0}(\Theta_{\bp}\vert_{\X}) \arrow[r, "df"] \arrow[d] & \H^{0}(\cn_{\X/\bp})  \arrow[d, "d\eta_{\ch}"]  \\
\H^{0}(\Theta_{\bg}\vert_{\F}) \arrow[r, swap, "d\sigma_{f}"] & \H^{0}(\cn_{\F/\bg}),
\end{tikzcd}
\end{equation*}
where the isomorphism on the left is induced by Chow's isomorphism $\Aut(\bp)\rightarrow\Aut(\bg)$, and \Cref{cubic} (ii), 
\Cref{fano} (ii). We expect $\eta$ to be an isomorphism without assuming the condition $\H^{0}(\Theta_{\X})=0$.
\end{remark}

\subsection{Further questions}

There are a number of follow-up questions. If $\X$ is a Lefschetz cubic with a node at $x_{0}$, 
then the singular locus of $\F$ can be identified with a smooth complete intersection $\Sigma\subset\bp_{d}$ of type $(2,3)$. 
The scheme $\F$ has rational singularities, and the blow-up
\begin{equation*}
\tilde{\F}\rightarrow\F
\end{equation*}
of $\F$ along $\Sigma$ provides a resolution of singularities of $\F$ \cite[Theorem 7.8]{ClG}. 
In such a situation, a general construction of Wahl \cite[Remark 1.4.1]{Wah} yields a blow-down morphism
\begin{equation*}
\beta:\cd_{\tilde{\F}}\rightarrow\cd_{\F}.
\end{equation*}
Here $\tilde{\F}$ is closely related to the Hilbert scheme of points $\Sigma^{[2]}$. 
By \cite[Theorem 36]{BFR}, one has a canonical isomorphism $\H^{1}(\Theta_{\Sigma})\xrightarrow{\sim}\H^{1}(\Theta_{\Sigma^{[2]}})$, 
which shows in particular that $\H^{1}(\Theta_{\Sigma^{[2]}})$ has dimension $\binom{d+2}{3}$; 
since this is also the dimension of $\Ext^{1}(\Omega^{1}_{\F}, \co_{\F})$, the morphism $\beta$ might be an isomorphism.

On the other hand, for smooth $\X$ it would be interesting to relate the non-commutative deformation theory (in the sense of \cite{Tod}) of $\X$ to the one of $\F$. 
A crucial role is played by the Hochschild cohomology
\begin{equation*}
\HH^{2}(\F) = \H^{0}(\Lambda^{2}\Theta_{\F})\oplus\H^{1}(\Theta_{\F}),
\end{equation*}
and the first step in this direction would be to compute the space $\H^{0}(\Lambda^{2}\Theta_{\F})$ of bivector fields on $\F$, and to exhibit Poisson structures on $\F$.

\subsection*{Acknowledgments}

We are indebted to P. Belmans, C. Borcea, F. Catanese, D. Huybrechts, S. Kleiman, and R. Thomas for helpful correspondence. 
The work on this paper was started when the author was at ETH Zurich, supported by the grant SNF-200020-182181; 
it was completed when the author was a Simons Foundation research fellow at the the Isaac Newton Institute for Mathematical Sciences,
funded by Simons Foundation grant 817244 and EPSRC grant EP/R014604/1.

\noindent 
Isaac Newton Institute, University of Cambridge\\
Departement Mathematik, ETH Zürich

\end{document}